\documentclass{amsart}

\usepackage{hyperref, color}
\usepackage{amssymb}
\usepackage{epsf}
\usepackage{graphicx}
\usepackage{graphics}
\usepackage{mathrsfs}

\theoremstyle{plain}
\newtheorem{theorem}{Theorem}[section]
\newtheorem*{theorem*}{Theorem \ref{thm:appl}}

\newtheorem{lemma}[theorem]{Lemma}

\theoremstyle{definition}

\newtheorem{remark}[theorem]{Remark}
\newtheorem{definition}[theorem]{Definition}

\numberwithin{equation}{section}

\newcommand{\abs}[1]{\lvert#1\rvert}
\newcommand{\norm}[1]{\|#1\|}

\newcommand{\field}[1]{\mathbb{#1}}
\newcommand{\real}{\field{R}}

\newcommand{\meta}[2]{\langle #1,#2 \rangle }

\newcommand{\R}{\mathbb R}
\newcommand{\s}{\mathbb S}
\newcommand{\Hy}{\mathbb H}

\newcommand{\ov}{\overline}
\newcommand{\set}[1]{\left\{#1\right\}}

\begin{document}

\title[Min-Oo conjecture for fully nonlinear equations]
{Min-Oo conjecture for fully nonlinear conformally invariant equations}

\author[Barbosa]{Ezequiel Barbosa}     
\address{Departamento de Matem\'atica, Universidade Federal de  Minas Gerais,  Belo Horizonte-Brazil}
\email{ezequiel@mat.ufmg.br}

\author[Cavalcante]{Marcos P. Cavalcante}     
\address{Instituto de Matem\'atica, Universidade Federal de  Alagoas,  Macei\'o-Brazil}
\email{marcos@pos.mat.ufal.br}

 \author[Espinar]{Jos\'e  M. Espinar}
 \address{Departamento de Matem\'aticas, Faculad de Ciencias, Universidad de C\'adiz }
 \address{Instituto Nacional de Matem\'atica Pura e Aplicada, Rio de Janeiro - Brazil}
 \email{jespinar@impa.br}

\date{}

\thanks{The second author is partially supported by CNPq-Brazil (Grant 309543/2015-0), CAPES-Brazil (Grant 897/18) and FAPEAL-Brazil (Projeto Universal).
The third author is partially supported by Spanish MEC-FEDER (Grant MTM2016-80313-P and Grant RyC-2016-19359); CNPq-Brazil (Grant 402781/2016-3 and Grant 306739/2016-0); FAPERJ-Brazil (Grant 232799).}
\begin{abstract} 
In this paper we show rigidity results for super-solutions to fully nonlinear elliptic conformally invariant equations on subdomains of the standard $n$-sphere $\s^n$ under suitable conditions along the boundary. We emphasize that our results do not assume concavity assumption on the fully nonlinear equations we will work with.

This proves rigidity for compact connected locally conformally flat manifolds $(M,g)$ with boundary such that the eigenvalues of the Schouten tensor satisfy a fully nonlinear elliptic inequality and whose boundary is isometric to a geodesic sphere $\partial D(r)$, where $D(r)$ denotes a geodesic ball of radius $r\in (0,\pi/2]$ in $\s^n$, and totally umbilical with mean curvature bounded below by the mean curvature of this geodesic sphere. Under the above conditions, $(M,g)$ must be isometric to the closed geodesic ball $\overline{D(r)}$. 

As a side product, in dimension $2$ our methods provide a new proof to Toponogov's Theorem about the rigidity of compact surfaces carrying a shortest simple geodesic. Roughly speaking, Toponogov's Theorem is equivalent to a rigidity theorem for spherical caps in the Hyperbolic three-space $\mathbb H^3$.  In fact, we extend it to obtain rigidity for super-solutions to certain Monge-Amp\`ere equations.
\end{abstract}

\maketitle

\section{Introduction}

In 1995, Min-Oo \cite{O}, inspired by the work of Schoen and Yau \cite{SY1,SY2} on the Positive Mass Theorem, conjectured that if $(M^n,g)$
is a compact Riemannian manifold with boundary such that the scalar curvature of $M$ is at least $n(n-1)$ and whose boundary $\partial M$ is totally geodesic and
isometric to the standard sphere, then $M$ is isometric to the closed hemisphere 
$\overline{\mathbb S^n_+}$  equipped with the standard round metric.  
Analogous statement of the Min-Oo conjecture for $\mathbb R^n$ 
(instead for $\overline{\mathbb S^n_+}$  as the original conjecture above) 
was proved in 2002 (see \cite{Miao} and \cite{ST}).
On the other hand, a counterexample for the Min-Oo conjecture was given by Brendle, Marques and Neves in 2011  in \cite{BMN} . 

The Min-Oo conjecture on $\overline{\mathbb S^n_+}$ among metrics conformal to the standard metric on the hemisphere was proved by Hang and Wang in  \cite{HW}. Namely: 

%

\begin{theorem} [Hang-Wang \cite{HW}]\label{HW}
Let $g=e^{2\rho}g_0$ be a $C^2$ metric on the unit closed hemisphere $\overline{\mathbb S^n_+}$, where $g_0$ denotes the standard round metric. Assume that  
\begin{itemize}
\item[(a)] $R_g\geq n(n-1)$, and
\item[(b)] the boundary is totally geodesic and isometric to the standard $\mathbb S^{n-1}$.
\end{itemize}
Then $g$ is isometric to $g_0$. 
\end{theorem}
We point out here that Hang and Wang also established  
a Ricci curvature version of the Min-Oo conjecture in  \cite{HW2}.

Recently, Spiegel \cite{SP} showed a scalar curvature rigidity theorem for locally conformally flat manifolds with boundary in the spirit of Min-Oo's conjecture which is an extension of Hang-Wang's Theorem. To be more precise, let $p\in \mathbb S^n$, $0< r\leq\frac{\pi}{2}$ and 
\[
D(p,r):=\{x\in\mathbb S^n \, : \,\, d_{g_0}(x,p)<r\}
\]
be the geodesic ball of radius $r$ centered at $p$ in $\mathbb S^n$. Let $H_{r}=\cot(r)$ be the mean curvature of the boundary $ \partial D(p,r)$, measured with respect to the inward orientation. Note that $\partial D(p,r)$ is isometric to a sphere of radius $\sin(r)$. 

\begin{theorem} [Spiegel \cite{SP}]\label{Sp}
Let $(M^n,g)$, $n\geq3$, be a compact connected locally conformally flat Riemannian manifold with boundary. Assume that   
\begin{itemize}
\item[(a)] $R_g\geq n(n-1)$, and
\item[(b)] the boundary $\partial M$ is umbilic with mean curvature $H_g \geq H_r$ and isometric to $\partial D(p,r)$, $0< r \leq \pi/2$. Here, the mean curvature is measured with respect to the inward orientation. 
\end{itemize}
Then $(M,g)$ is isometric to $\overline{D(p,r)}$ with the standard metric.
\end{theorem}

\begin{remark}
Spiegel also proved that the assumption on the mean curvature in the theorem above can be dropped provided $M$ is simply-connected and $r=\frac{\pi}{2}$. See Remark 1.3 in \cite{SP}. Therefore, Theorem \ref{Sp} is an extension of Theorem \ref{HW}.
\end{remark} 

Theorem \ref{Sp} is sharp in $r$ in the sense that one can construct counterexamples on $\overline{D(p,r)}$ for $\pi/2 < r <\pi $ (cf. \cite{HW}).

We are interested in the Min-Oo's conjecture for compact connected locally conformally flat Riemannian manifolds $(M^n,g)$ satisfying a more general curvature condition. It is well known that the scalar curvature is, up to a constant, the sum of the eigenvalues of the Schouten tensor ${\rm Sch}_g$. In fact,  let $\lambda (p)=(\lambda_1(p),\ldots, \lambda_n (p))$  denote its  eigenvalues, then
\begin{equation}\label{Eq:ScalarSchouten}
{\rm Trace}(g^{-1}{\rm Sch}_g)=\lambda_1(p)+\cdots +\lambda_n(p)=  \frac{R(g)}{2 (n-1)}\,.
\end{equation}

It is natural to ask if the Min-Oo's conjecture holds when one considers a more general function on the eigenvalues of the Schouten tensor instead of the scalar curvature. 
In order to establish properly our main result, we need to define the type of curvature function for the eigenvalues of the Schouten tensor that we will consider. First, let us recall the  notion of elliptic data originally introduced by Caffarelli, Nirenberg and Spruck \cite{CNS}; we use the theory developed by Li and Li for conformal equations (cf.  \cite{LiLi1,LiLi2}). Consider the convex cones
\begin{equation*}
\begin{split}
  \Gamma_{n} =& \{x\in\R^{n} \, :  \, \, x_{i}>0, \,\, i=1,\ldots,n\}, \\
  \Gamma_{1}= & \left\{ x\in\R^{n} \, : \,\, x_{1}+\cdots+x_{n}>0\right\} .
\end{split}
\end{equation*}

Let $\Gamma\subset\R^{n}$ be a symmetric open convex cone and  $f\in C^{1}\left(\Gamma\right)\cap C^0\left(\ov{\Gamma}\right)$. We say that $(f,\Gamma)$ is an {\it elliptic data} if the pair $(f,\Gamma)$ satisfies
\begin{enumerate}
  \item $ \Gamma_{n} \subset \Gamma \subset \Gamma_{1} $,
  \item $f$ is symmetric,
  \item $f>0$ in $\Gamma$,
  \item $f|_{\partial\Gamma}=0$,
  \item $f$ is homogeneous of degree 1,
  \item $\nabla f(x)\in\Gamma_{n}$ for all $x\in\Gamma$,
  \item $f(1,\ldots ,1)= 2 $.
\end{enumerate}

Let $(M,g) $ be a Riemannian manifold. Then, given an elliptic data $(f, \Gamma )$ we say that $g$ is a {\it supersolution} to $(f,\Gamma)$ if 
\begin{equation*}
  f(\lambda _g(p)) \geq 1, \, \, \lambda _g (p) \in \Gamma \text{ for all } p \in M ,
\end{equation*}where $\lambda _g(p)=(\lambda_1(p),\ldots, \lambda_n (p) )$ is composed by the eigenvalues of the Schouten tensor of $g$ at $p \in M$.

It is well-known that the Schouten tensor of the standard $n$-sphere is  $ {\rm Sch}_{g_0} = \frac{1}{2} g_0 $, then, condition (7) above says that we are normalizing the functional $f$ to be $1$ when considering the Schouten tensor of the standard sphere, i.e., 
$$ f(1/2 , \ldots , 1/2) =  2^{-1}f(1,\ldots ,1) = 1 ,$$ where we have used that $f$ is homogeneous of degree one.

In this paper, we prove that the Min-Oo's conjecture holds for super-solutions to elliptic 
data $(f, \Gamma)$ in locally conformally flat manifolds. 
Namely, we prove the following result.

\begin{quote}
{\bf Theorem A.} \label{theoremA} {\it Let $(M^n,g)$  be a compact connected locally conformally flat Riemannian manifold with boundary $\partial M$. Let $(f, \Gamma)$ be an elliptic data and assume that $g$ is a supersolution to $(f,\Gamma)$ in $M$, i.e., 
$$f(\lambda _g(p)) \geq 1, \, \, \lambda _g(p)\in\Gamma \text{ for all } p\in M .$$

Assume that $\partial M$ is umbilical with mean curvature $H_g \geq H_r $ and  isometric to $\partial D(p,r)$, $0<r\leq \pi/2$. Then $(M,g)$ is isometric to $\overline{D(p,r)}$ with the standard metric.

} 
\end{quote}
\begin{remark}
We also can prove that the assumption on the mean curvature in the theorem above can be dropped provided $M$ is simply-connected and $r=\frac{\pi}{2}$.
\end{remark} 
\medskip

We emphasize that in our theorem above \emph{no concavity assumption on $f$ is needed}.
Of special interest is when we consider $\sigma_k(\lambda(p))$, the $k$-th elementary symmetric polynomial of the eigenvalues $\lambda_1 (p)$,...,$\lambda_n (p)$. However, 
these cases, and in fact for all concave $f$ ($\sigma_k^{1/k}$ is concave), the result 
follows  from the theorem of Spiegel. Indeed, we only need to 
prove that under the additional concavity assumption of $f$ in $\Gamma$, one has
\[
f(\lambda) \leq R_g/[n(n-1)], \textrm{ for all } \lambda\in \Gamma.
\]

The above inequality can be proved as follows. By the homogeneiety of $f$, 
$\sum f_{\lambda_i}\lambda_i=  f(\lambda)$ and therefore 
$2=f(1,\ldots, 1)$ and, in view of the symmetry of $f$, 
$f_{\lambda_i}(1, \ldots, 1) = \frac 2 n$, $i=1, \ldots, n$. By the concavity of $f$ we get
\[
f(\lambda) \leq f(1,\ldots, 1) + \sum_{i=1}^n f_{\lambda_i}(\lambda_i-1) = R_g/[n(n-1)].
\]

Our approach relies in a geometric  method developed by the third author, G\'alvez and Mira 
in \cite{EGM} and further developments contained in \cite{AE,BEQ,BQZ,CE,Esp}, where 
conformal metrics on spherical domains are represented by hypersurfaces in the hyperbolic space. 
In order to reduce our problem on locally conformally flat manifolds to conformal metrics on 
subdomains of the sphere, we use results contained in the work of Spiegel \cite{SP} and  Li and 
Nguyen \cite{LiNg} based on the deep theory by Schoen and Yau \cite{SY3} on the developing map 
of a locally conformally flat manifold. Hence, combining these results, we show that Theorem A is 
equivalent to a rigidity result for horospherically concave hypersurfaces with boundary in the 
Hyperbolic space $\mathbb H^{n+1}$. In particular, in dimension $n=2$, these methods provide a 
new proof to Toponogov's Theorem \cite{Top} and, in fact, we can extend it.

\subsection*{Acknowledgments}
The authors are grateful to the referee for him/her
valuable comments and suggestions that have improved this article.

\section{Preliminaries}\label{elliptic}

We will establish in this section the necessary tools we will use along this paper.

\subsection{Representation formula and regularity}

Here we recover the hypersurface interpretation of conformal metrics on the sphere developed in 
\cite{BEQ,EGM}.  Let us denote by $\mathbb{L}^{n+2}$ the Minkowski spacetime, that is, the vector 
space $\R ^{n+2}$ endowed with the Minkowski spacetime metric $\meta{}{}$ given by
$$ 
\meta{\bar{x}}{\bar{x}} = - x_0 ^2 + \sum _{i=1}^{n+1} x_i ^2,
$$
where $\bar{x} \equiv (x_0 , x_1 , \ldots , x_{n+1})\in \R^{n+2}$. 

Then hyperbolic space, de Sitter spacetime and positive null cone are given, respectively, by the hyperquadrics
\begin{equation*}
\begin{split}
\mathbb{H} ^{n+1} &= \set{ \bar{x} \in \mathbb L ^{n+2} : \, \meta{\bar{x}}{\bar{x}} = -1, \, x_0 >0}\\
d\s^{n+1}_1 &= \set{ \bar{x} \in \mathbb L ^{n+2} : \, \meta{\bar{x}}{\bar{x}} = 1}\\
\mathbb{N}^{n+1}_+ &= \set{ \bar{x} \in \mathbb L ^{n+2} : \, \meta{\bar{x}}{\bar{x}} = 0, \, x_0 >0}.
\end{split}
\end{equation*}

Let $\phi:M^n\to \mathbb{H}^{n+1} \subset \mathbb{L} ^{n+2}$ be an isometric immersion of an oriented hypersurface, with orientation $\eta :M^n\to d\s^{n+1}_1 \subset \mathbb{L} ^{n+2}$.  We define the associated light cone map as
$$\psi := \phi - \eta : M^n \to \mathbb{N} ^{n+1}_+ \subset \mathbb{L} ^{n+2} .$$

If we write $\psi = (\psi _0 , \ldots , \psi _{n+1})$, consider the map  $G$ (the hyperbolic Gauss map) given by:
$$ G = \frac{1}{\psi _0}(\psi _1 , \ldots , \psi _{n+1}) : M \to \s ^n, $$

Hence, if we label $e^{\rho}:=\psi _0$ (the hyperbolic support function), we get
$$ \psi = e^{\rho} (1 , G) \in \mathbb{L} ^{n+2}.$$

Set $\Sigma := \phi(M^n)\subset \mathbb H^{n+1}$ with orientation $\eta$. We say that $\Sigma$ is horospherically concave if $\Sigma$ lies (locally) around any point $p\in \Sigma$ strictly in the concave side of the tangent horosphere at $p$ and its normal points into the concave side of the tangent horosphere.

\begin{theorem}[\cite{EGM}]
Let $\phi:\Omega\subset \mathbb S^n \to \mathbb H^{n+1}$ be an oriented piece of horospherically concave hypersurface with orientation $\eta : \Omega \to d\s ^{n+1}_+$ and hyperbolic Gauss map $G(x) = x$. Then
\begin{equation}\label{formula}
\phi(x) = \frac{e^\rho}{2}\big(1+e^{-2\rho}(1+ \|\nabla \rho  \|^2)\big)(1,x)+e^{-\rho}(0,-x+\nabla\rho), 
\end{equation}and its orientation is given by 
\begin{equation}\label{orientation}
\eta (x) = \phi (x) - e^{\rho}(1,x). 
\end{equation}

Moreover, the eigenvalues $\lambda_i$ of the Schouten tensor of $g=e^{2\rho}g_0$ and  the principal curvatures $k_i$ of $\phi$ are related by
$$\lambda_i=\frac 1 2 - \frac 1 {1+k_i}.$$

Conversely, given a conformal metric $g = e^{2\rho}g_0$  defined on a domain of the sphere $\Omega\subset \mathbb S^n$  such that the eigenvalues of its Schouten tensor are all less than $1/2$,  then the map $\phi$ given by \eqref{formula} defines an immersed, horospherically concave hypersurface in $\mathbb H^{n+1}$ with orientation \eqref{orientation} whose hyperbolic Gauss map is $G(x) = x$ for $x \in \Omega$.

Here, the connection $\nabla $ and the norm $\norm{\cdot} $ are with respect to the standard metric $g_0$ on $\s ^n$.
\end{theorem}

Let $\Omega \subset \s ^n$ be a relatively compact domain with smooth boundary. Given $\rho \in C^{2} (\overline{\Omega})$, the above representation formula says that $\phi$ and $\eta$ are $C^1$ maps and $\Sigma := \phi (\overline \Omega) \subset \mathbb{H}^{n+1}$ is a compact hypersurface with boundary $\partial \Sigma = \phi(\partial \Omega)$ whose tangent plane varies $C^1$. Moreover, the corresponding conformal metric $g= e^{2\rho} g_0$ on $\Omega$ is the horospherical metric associated to $\Sigma$. Observe that, since $\rho \in C^2(\overline \Omega)$, the eigenvalues of the Schouten tensor associated to $g=e^{2\rho} g_0$ are continuous in $\Omega$ and hence there exists $t >0$ so that the eigenvalues of the Schouten tensor associated to $g_t = e^{2(\rho +t)} g_0$ are less than $1/2$.

In the Poincar\'{e} ball model of $\mathbb H ^{n+1}$, the representation formula (cf. \cite{AE}) is given by 
\begin{equation*}
  \varphi_{t}(x)=\frac{1-e^{-2\rho _t(x)}+\|\nabla e^{-\rho _t}(x) \|^{2}}{ \left( 1+e^{-\rho_t(x)} \right)^{2}+\|\nabla e^{-\rho_t}(x) \|^{2}}x-\frac{1}{ \left( 1+e^{-\rho_t(x)} \right)^{2}+\|\nabla e^{-\rho_t}(x) \|^{2}}\nabla\left( e^{-2\rho_t}\right)(x).
\end{equation*} 

Set $\epsilon = e^{-t}$, then 
$$f  (x, \epsilon):= - \frac{2(e^{\rho (x)}+ \epsilon )}{ \left( e^{ \rho(x) }+\epsilon \right)^{2}+ \epsilon ^2 \, \|\nabla \rho(x) \|^{2}}$$and 
$$ g (x,\epsilon) = \frac{2 \epsilon }{ \left( e^{ \rho(x) }+\epsilon \right)^{2}+\epsilon ^2 \, \|\nabla \rho(x) \|^{2}} $$are in $C^1(\overline{\Omega}\times [0,+\infty))$ and they are smooth in $\epsilon$, moreover, the vector field $\nabla \rho$ is $C^1$ in $\overline{\Omega}$, since $\rho \in C^2(\overline \Omega)$. Thus, 
$$ \varphi _\epsilon (x) = x + \epsilon  \left( f (x,\epsilon) x + g (x,\epsilon) \nabla \rho (x)\right) \in \mathbb{B}^{n+1} \subset \R ^{n+1} $$belongs to $C^1(\overline \Omega)$, in particular, the vector field 
$$ Y(x,\epsilon) :=  f (x,\epsilon) x + g (x,\epsilon) \nabla \rho (x) \in C^1(\overline \Omega \times [0,+\infty)).$$

Let $\tilde Y : \s ^n \times [0,+\infty ) \to \mathbb{B}^{n+1} \subset \R ^{n+1}$ be the Lipschitz extension of $Y $ so that 
$\left.\tilde Y  \right._{|\overline\Omega \times [0,+\infty)} =Y $. Therefore, the corresponding extension map 
$$ \tilde \varphi : \s ^n \times [0, +\infty ) \to \R^{n+2} $$is Lipschitz in $x$ and smooth in $\epsilon $ so that $\tilde \varphi (x,\epsilon) = \varphi_\epsilon (x)$ for all $(x,\epsilon ) \in \overline \Omega \times (0,+\infty)$ satisfying $\tilde \varphi (x,0) = x$, i.e., $\tilde\varphi _0 (\cdot)= \tilde \varphi (\cdot , 0)$ is the identity map, which is an embedding of the sphere $\s ^n$ into $\R ^{n+1}$. Since $\tilde \varphi _\epsilon : \s ^n \to \R ^{n+1}$ is a Lipschitz deformation of an embedding, from \cite{FukNak}, there exists $\epsilon _0 >0$ so that $\tilde \varphi _\epsilon : \s ^n \to \R ^{n+1} $ is an embedding for all $\epsilon \in [0, \epsilon _0 )$. Thus, summarizing all we have done in this subsection, we obtain:

\begin{lemma}[\cite{AE,BEQ,EGM}]\label{KeyLemma}
Let $\Omega \subset \s ^n$ be a relatively compact domain with smooth boundary and $\rho \in C^2 (\overline \Omega)$. Then, there exists $t >0$ so that the horospherically concave hypersurface $\phi _t : \overline \Omega \to \mathbb{H} ^{n+1}$ given by \eqref{formula} is a compact embedded hypersurface $\Sigma _t = \phi _t (\Omega)$ with boundary $\partial \Sigma _t = \phi _ t (\partial \Omega)$. Moreover, the eigenvalues of its associated horospherical metric $g_t := e^{2(\rho +t)}g_0$ are less than $1/2$.
\end{lemma}

It is important to recall the connection between isometries of the hyperbolic space ${\rm Iso}(\mathbb H^{n+1}) $ and conformal diffeomorphisms of the sphere ${\rm Conf}(\s^n)$. It is well-known that each isometry $T\in {\rm Iso}(\mathbb H^{n+1})$ induces a unique conformal diffeormorphism $\Phi \in {\rm Conf}(\s^n)$. 

Let  $T\in {\rm Iso}(\mathbb H^ {n+1})$ be an isometry and $\Phi \in {\rm Conf}(\s ^n)$ be the unique conformal diffeomorphism associated to $T$. Then, given a horospherically concave hypersurface $\Sigma \subset \mathbb H ^{n+1}$ with horospherical metric $g$, one can see that (cf. \cite{Esp}) the horospherical metric $\tilde g$ associated to $\tilde \Sigma = T(\Sigma)$ is given by $\tilde g = \Phi ^* g$. Vice versa, given a conformal metric $g$ on a subdomain of the sphere with associated hypersurface $\Sigma$, given by the representation formula under the appropriated conditions, the associated horospherically concave hypersurface $\tilde \Sigma $ associated to the conformal metric $\tilde g = \Phi ^* g$ is given by $\tilde \Sigma = T(\Sigma)$.

\subsection{Locally conformally flat metrics and developing map}

Let $(M^n,g)$, $n\geq 3$, be a Riemannian manifold with a $C^k$-metric $g$. We say that $(M,g)$ is locally conformally flat if for every point $p \in M$ there exist a neighborhood $U$ of $p$ and $\varphi \in C^k (U)$ such that the metric $e^{2\varphi} g$ is flat on $U$. An immersion $\Psi: (M, g) \rightarrow (N, h)$ is a conformal immersion if we can write $\Psi^* h = e^{2\varphi} g$ for some function $\varphi$. 

If $(M, g)$ is a locally conformally flat manifold it is well known that there exists a conformal map $\Psi : M \rightarrow \s^n$, called the {\it developing map} which is unique up to conformal transformations of $\s ^n$.  
When $M$ is compact and simply-connected with umbilical boundary, Spiegel \cite{SP} proved that the developing map can be taken as a diffeomorphism over the hemisphere $\overline{\mathbb S^n}_+$.

If $M$ is not simply-connected, we can pass to the universal covering $\tilde{M}$ to obtain a developing map $\Psi : \tilde{M} \rightarrow \s^n$ which is, under some assumptions, injective. In fact, Li and Nguyen \cite{LiNg} showed the following theorem:

\begin{theorem}\label{SP}
Let $(M, g)$ be a compact connected locally conformally flat manifold with boundary. Assume that $M$ has positive scalar curvature and that $\partial M$ is umbilic and simply-connected with non-negative mean curvature. Let $\Pi: \tilde{M}\rightarrow M$ be the universal covering. Then there exists an injective conformal map $\Psi : \tilde{M} \rightarrow \s^n$ which is a conformal diffeomorphism onto its image. The image is of the form
$$\Omega = \Omega (\epsilon_i , p_i , \Lambda) := \s ^n\backslash \left( \bigcup\limits_{i} D(p_i ,\epsilon_i)\cup \Lambda\right),$$
where the $D(p_i ,\epsilon_i)$ are geodesic balls in $\s^n$ centered at $p_i$ of radius $\epsilon _i$ with disjoint closures and $\Lambda$ is the so-called limit set, a closed subset of Hausdorff dimension at most $\frac{n-2}{2}$.
\end{theorem}

For the sake of completeness we include their proof here.

\begin{proof}
Actually, we can see that Theorem \ref{SP} is a consequence of Theorem 1.4 in \cite{LiNg}. In order to see that, note that one has an additional hypothesis that $\partial M$ is simply connected. Hence, the two points which need to be checked, under this additional hypothesis, are (1) the closed balls $\bar{D}(p_i ,\epsilon_i)$ in \cite{LiNg} are mutually disjoint and (2) the set $G=\s ^n\backslash \left( \bigcup\limits_{i} D(p_i ,\epsilon_i)\cup \Lambda\right)$ in \cite{LiNg} is simply connected. 

Point (1) is a consequence of Property (ii) in [15, Theorem 1.4] and some facts from point-set topology. First, $\Psi^{-1}(\partial M) =\cup_{i}(\partial D(p_i ,\epsilon_i) \setminus \Lambda)$ and, as $\Lambda$ is closed and its $(n - 2)$-Hausdorff measure is zero, $\partial D(p_i ,\epsilon_i) \setminus \Lambda$ is (path-)connected for every $i$. This implies, in view of Property (ii) in [15, Theorem 1.4], that the connected components of $\Psi^{-1}(\partial M)$ are the collection $\{\partial D(p_i ,\epsilon_i) \setminus \Lambda\}$. Second, for any connected component $X$ of $\Psi^{-1}(\partial M)$, the map $\Psi: X\rightarrow  \partial M$ is a covering map. Now if $\partial M$ is simply connected, each such $X$ is homeomorphic to $\partial M$, and so, $X$ is compact. Thus, $\partial D(p_i ,\epsilon_i) \cap \Lambda$ is empty for every $i$, and, in view of Property (ii) in [15, Theorem 1.4], the balls $\bar{D}(p_i ,\epsilon_i)$ are mutually disjoint.

Let us turn to point (2). If $\Lambda$ is empty, the collection $\{D(p_i ,\epsilon_i)\}$ of balls must be finite thanks to property (iii) in [15, Theorem 1.4], in which case the simple connectedness of $G$ is clear. Assume that $\Lambda$ is non-empty. Recall that $\Psi$ is constructed in \cite{LiNg} as the covering map from the universal cover $\tilde{M}_2 \subset \mathbb{S}^n$ of the double $M_2$ of $M$, still denoted by $\Psi$ here, and $G$ is a connected component of $\Psi^{-1}(M)$. Let $\Lambda_2:=\mathbb{S}^n\setminus \tilde{M}_2$ so that $\Lambda=\Lambda_2 \cap(\mathbb{S}^n \setminus (\cup D(p_i ,\epsilon_i)))$.

Suppose first that $D(p_i ,\epsilon_i) \cap\Lambda_2\neq\emptyset$ for every $i$. In this case, as $G=\tilde{M}_2 \setminus(\cup D(p_i ,\epsilon_i))$, $\partial D(p_i ,\epsilon_i)\subset G$ (due to the simple connectedness of $\partial M$ as in point (1)) and $D(p_i ,\epsilon_i) \cap (\mathbb{S}^n \setminus \tilde{M}_2)\neq \emptyset$ for each $i$, there is clearly a retraction from $\tilde{M}_2$ onto $G$. The simple connectedness of $G$ follows from that of $\tilde{M}_2$.

Assume now that $D(p_{i_0} ,\epsilon_{i_0})\cap \Lambda_2=\emptyset$ for some $i_0$. We have $\Psi^{-1}(M_2 \setminus M) \subset \cup(D(p_i ,\epsilon_i) \setminus \Lambda_2)$. Hence, as $\Psi$ is locally homeomorphic and by Property (iii) in [15, Theorem 1.4], 
\[
 \Psi^{-1}(\overline{M_2 \setminus M}) \subset \overline{\Psi^{-1}(M_2 \setminus M)}\cap(\mathbb{S}^n \setminus \Lambda_2)\subset \cup (\bar{D}(p_i ,\epsilon_i)\setminus \Lambda_2).\,
\]
As the balls $\bar{D}(p_i ,\epsilon_i)$ are disjoint, the above implies that there is a connected component of $\Psi^{-1}(\overline{M_2 \setminus M})$ lying entirely in $D(p_{i_0} ,\epsilon_{i_0})$, which covers $\overline{M_2 \setminus M}$, which is a copy of $M$. We can use this set in place of the original set $G$ to run the argument, in which case $\Lambda\subset D(p_{i_0} ,\epsilon_{i_0})\cap\Lambda_2=\emptyset$ becomes empty and we are done as above.
\end{proof}

Note that, since we are assuming $\lambda_g(p)\in \Gamma$, for all $p\in M$, and $\Gamma \subset \Gamma_1$, hence we have that $R_g>0$. Therefore, under the conditions of Theorem A, we can apply Theorem \ref{SP}.

\section{The case of the hemisphere}\label{hemisphere}

We begin by considering the baby case, say conformal metrics on the hemisphere. This case will enlighten the geometric ideas contained in the proof. 

\begin{theorem}\label{Th:Hemi}
Let $(f, \Gamma)$ be an elliptic data and let $g=e^{2\rho}g_0$,  $\rho \in C^2 (\overline{\mathbb S^n_+})$, $n\geq 3$, be a supersolution to $(f,\Gamma)$ on the closed  hemisphere $\overline{\mathbb S^n_+}$, i.e., 
$$f(\lambda(p)) \geq 1, \, \, \lambda _g(p)\in\Gamma \text{ for all } p\in \mathbb{S}^n_+ .$$

Assume that the boundary $\partial \mathbb S^n_+$ with respect to $g$ is isometric to  $\partial \mathbb{S}^n_+$. Then $g=\Phi^*g_0$, where $\Phi \in {\rm Conf}(\s^n)$ preserving $\overline{\s ^n _+}$.
\end{theorem}
\begin{proof}
First, $\partial \s ^n _+$ is isometric to $\s ^{n-1}$ implies that $g_{|\partial \mathbb S ^n_+}$ is isometric to $\s ^{n-1}$. Hence, by Obata's Theorem, there exists a conformal diffeomorphism $\tilde \Phi  \in {\rm Conf}(\s ^{n-1})$ so that $g_{|\partial \mathbb S ^n_+}= \tilde \Phi ^* \left.g_0\right. _{|\partial \mathbb S ^n_+}$ along $\partial \s ^n _+$. Observe that $\tilde \Phi$ can be extended to a conformal diffeomorphism $\Phi \in {\rm Conf}(\s^n)$ so that $\Phi (\s^n _+) = \s ^n _+$ and $\Phi _{| \partial \s ^n_+} = \tilde \Phi$. Hence, up to the conformal diffeomorphism $\Phi$, we can assume that $g= g_0$ along $\partial \s ^n _+$. In other words, 
\begin{equation}\label{RB}
\rho = 0 \text{ on } \partial \s ^n _+ .
\end{equation}

Moreover, since $\partial \s ^n _+$ is totally geodesic with respect to $g_0$ and $g$ is conformal to $g_0$, $\partial \s ^n _+$ is totally 
umbilical with respect to $g$, in particular, the mean curvature along $\partial \s ^n _+$ with respect to $g$ is given by
\begin{equation}\label{MB}
H_g:= -e^{-\rho}\frac{\partial \rho}{\partial \nu} = -\frac{\partial \rho}{\partial \nu} \text{ on } \partial \s ^n _+ ,
\end{equation}where $\nu = e_{n+1}$ is the inward normal along $\partial \mathbb S ^n_+$. 

Let $P \subset \mathbb{H}^{n+1}$ be the totally geodesic hyperplane whose boundary at infinity is the equator of the upper hemisphere, i.e., $\partial _\infty P = \partial \s ^n _+$. Denote by $P^+$ (resp. $P^-$) the connected component of $\mathbb{H}^{n+1}\setminus P$ that contains the north pole (resp. south pole) at its boundary at infinity. Also, denote by $P(b)$, $b\in \R$, the equidistant to $P$ at distance $b$. Note that $P(b)\subset P^+$  when $b>0$ and $P(b)\subset P^-$ when  $b<0$. We define $P(b)^+$ (resp. $P(b)^-$) as the connected component of $\mathbb H ^{n+1}\setminus P(b)$ containing the north pole (resp. south pole) in its boundary at infinity. Clearly, $\partial _\infty P(b) = \partial _\infty P = \partial \mathbb S ^n _+ $ for all $b \in \R$.

Now, we fix $t>0$ as in Lemma  \ref{KeyLemma} such that  the eigenvalues of the Schouten  tensor of  $g_t=e^{2(\rho +t)}g_0$  satisfy $\lambda^t_i(x)<1/2$ for all $x\in \overline{\mathbb S^n_+}$ and the compact horospherically concave hypersurface with boundary $\Sigma _t=\phi _t(\mathbb S^n_+) \subset \mathbb H^{n+1} \subset \mathbb L ^{n+2}$  given by the representation formula (\ref{formula}) associated to $\rho _t =\rho +t$ is embedded. 
Given $p\in  \mathbb{H}^{n+1}$ we denote by $d_{\mathbb H^{n+1}}(p, P)$  the signed distance to $P$, that is,  it is positive if $p \in P^+$ and negative if $p\in P^-$. Then, taking $t>0$ big enough in Lemma \ref{KeyLemma} we can assume that $\Sigma _t$ is above $P(m)$, i.e., $\Sigma _t \subset \overline{P(m)^+}$, where $m = {\rm min}\set{ d_{\mathbb H^{n+1}}(p, P) \, : \, \, p \in \partial \Sigma _t} $. In fact, one can check (cf. \cite[Section 2.4]{AE} for details) that  $m= {\rm min}\set{ {\rm arc}\sinh (-e^{-t}H_g(x)) \, : \,\, x \in \partial \mathbb S ^n _+ } $.

Observe that \eqref{RB} implies
\begin{equation}\label{Rt}
\rho_t = t \text{ and } \frac{\partial \rho _t}{\partial \nu} =  \frac{\partial \rho }{\partial \nu}  \text{ on } \partial \s ^n _+ .
\end{equation}

We claim:

\begin{quote}
{\bf Claim A:} {\it Let $\gamma : \R \to \mathbb H^{n+1}$ be the complete geodesic (parametrized by arc-length) joining the south and north poles. Let $\mathcal C _t$ be the solid cylinder in $\mathbb H ^{n+1}$ of axis $\gamma$ and radius $t$. Then, $\partial \Sigma _t$ lies outside the interior of $\mathcal C _t$, and $\partial \Sigma _t \cap \mathcal C _t \subset P$. Moreover, if  $\partial \Sigma _t \cap \mathcal C _t \neq \emptyset$ then at such points $\Sigma _t$ is orthogonal to $P$.}
\end{quote}
\begin{proof}[Proof of Claim A]

Note that, since $x\in \partial \mathbb S^n_+$, $\phi _t(x) \in \mathcal H(x,t)$, where $\mathcal H(x,t)$ is the horosphere whose point at 
infinity is  $x$ and signed distance to the origin is $t>0$ (see \cite{BEQ}). It proves the first part of the claim 

To finish the proof, we must check that at a point where $\frac{\partial \rho}{\partial \nu} (x) =0$ we get that $\Sigma _t$ is orthogonal to $P$. The unit normal along $\Sigma_t$ is given by
$$ \eta _t(x) = \frac{e^{-\rho-t}}{2} \big(\norm{\nabla \rho }^2 -1+e^{\rho+t}  \big) (1,x) + e^{-\rho-t}(0,-x+ \nabla \rho) $$and the normal along $P$ is given by $ n(p) = (0,e_{n+1}) $ for all $ p \in P$. Hence, we have
$$ \meta{\eta_t (x)}{n(\phi (x))} = 0 ,$$that is, $\Sigma _t$ is orthogonal to $P$ at $x$.
\end{proof}

Let $(1, {\bf 0}):=(1,0, \ldots , 0) \in \mathbb H ^{n+1} \subset \mathbb L ^{n+2}$ be the origin in the hyperboloid model  (note that such point corresponds to the actual origin in the Poincar\'e ball model). Denote by $S_t \subset \mathbb{H}^{n+1}$ the geodesic sphere centered at the origin $(1,{\bf 0})$ of radius $t$. 

It is easy to see that its horospherical metric is given by $\tilde g _ t = e^{2t} g_0$ (cf. \cite{Esp}). Consider the half-sphere $S_t ^+ = S_t \cap \overline{P^+}$ and observe that  $S_t ^+$ is orthogonal to $P$ along the boundary $\partial S_t^+$. 

Let $T_s :\mathbb H ^{n+1} \to \mathbb H ^{n+1}$ be the hyperbolic translation at distance $s$ along $\gamma$ so that $T_s((1, {\bf 0}))= \gamma (s)$, an isometry of $\mathbb H ^{n+1}$. It is clear that $T_s(S_t^+\setminus \partial S_t ^+) \cap \partial \Sigma _t = \emptyset $,  for all $s \in \R$ by Claim A.

Let $\Phi _s \in {\rm Conf}(\s ^n)$ be the unique conformal diffeomorphism associated to $T_s$. Set $S_{t,s} := T_s (S_t)$ for all $s \in \R$, then the horospherical metric associated to $S_{t,s}$ is given by $\tilde g _{t,s} = e^{2t} \Phi _s ^* g_0$ in $\s ^n$ and denote by $\tilde \rho _{t,s} \in C^{\infty}(\s ^n )$ the horospherical support function associated to $S_{t,s}$, i.e, $\tilde g _{t,s}= e^{2\tilde \rho _{t,s}} g_0$.
Let $\hat g_{t,s}$ be the restriction of $\tilde g _{t,s}$ to $\overline{\s ^n_+}$, i.e., $\left.\tilde g _{t,s}\right. _{|\s ^n _+} = \hat g_{t,s}$, and $\hat \rho _{t,s}$ the restriction of $\tilde \rho _{t,s}$ to $\overline{\s ^n _+}$.

Consider $\bar s\in \R$ so that $S_{t,s} ^+ \cap \Sigma _t= \emptyset$ for all $s < \bar s$. Increasing $s$ from $\bar s$ to $+\infty$, we must find a first instant $s_0$ so that $S_{t,s_0} ^+ \cap \Sigma _t\neq \emptyset$ tangentially. If  $S_{t,s_0} ^+ $ does not coincides with  $\Sigma _t$ identically, such tangential point must be either at an interior point of $\Sigma _t$ or at a boundary point of $\partial \Sigma_t$. In the latter case we must  necessarily have $s_0=0$ by the second part of  Claim A.

\begin{quote}
{\bf Claim B:} {\it $\rho _t \geq \hat \rho _{t,s_0}$ on $\overline{\s ^n _+}$.}
\end{quote}
\begin{proof}[Proof of Claim B]
From Claim A we have that  $\mathcal H (x,r)$ either does not touch $S_{t,s_0}$ or does touch at a tangent point, for all $x \in \partial \s ^n _+$ and  all $r\geq t$. This says that $\rho _t \geq \hat \rho _{t,s_0}$ on $\partial \s ^n _+$ because $\Sigma_t$ is horospherically concave. 
Now, let us prove that $\rho _t \geq \hat \rho _{t,s_0}$ on $\s ^n _+$. Assume there exists $x \in \s ^n _+$ so that $\rho _t (x)< \hat \rho _{t,s_0}(x)$. Then, as pointed out above, the horosphere $\mathcal H(x,\hat \rho _{t,s_0}(x))$ does not touch $\Sigma _t$ and touch at one point $q\in S _{t,s_0}$. Observe that $\mathcal H(x,\hat \rho _{t,s_0}(x)- \delta) $ does not touch $\Sigma _t$ for any $\delta <  \hat \rho _{t,s_0}(x)-\rho _t (x)$. Denote by $\beta _1 $ the geodesic ray joining $q$ and the point at infinity $x \in \s ^n _+$, this arc is completely contained in the horoball determined by $\mathcal H(x,\hat \rho _{t,s_0}(x))$ and hence $\beta _1 \cap \Sigma _t =\emptyset$. Denote by $\beta _2$ the geodesic joining $\gamma (s_0)$ with the south pole ${\bf s}\in \s ^n$, then $\Sigma _t \cap \beta _2 = \emptyset$, otherwise we contradict the fact that $S_{t,s_0}$ is the first sphere of contact with $\Sigma _t$ coming from infinity. Finally, denote by $\beta _3$ the geodesic arc joining $\gamma (s_0)$ and $q$. Consider the piecewise smooth curve $\beta = \beta_1 \cup \beta _2 \cup \beta _3$ and observe that $\beta $ is homotopic to $\gamma$, moreover, $\partial \Sigma _t$ is homotopic to $\partial \s ^n _+$, which implies that the linking number of $\beta $ and $\partial \Sigma _t$ is $\pm 1$ (depending on the orientation), that is, they must intersects. The only possibility is that they intersect in the interior of $\beta _2$, however, this implies that $\Sigma _t$ and $S_{t,s_0}$ has a transverse intersection, contradicting that $S_{t,s_0} $ is the first sphere of contact. Thus, $\rho _t \geq \hat \rho _{t,s_0}$ on $ \s ^n _+$.
\end{proof}

Note that, since the elliptic data is homogeneous of degree one, we have that $g_t$ satisfies 
$$ f(\lambda _{g_t}(p)) = f (e^{-t}\lambda _g(p))\geq e^{-t} \text{ for all } p\in \s^n_+$$and the horospherical metric of $S_{t,s} ^+$ satisfies
$$f(\lambda _{\hat g _{t,s}}(p))= f(e^{-t}\lambda _{g_0}(p)) = e^{-t}f(1/2,\ldots,1/2)=e^{-t} \text{ for all } p\in \s^n_+,$$that is
$$  f(\lambda _{g_t}(p)) \geq  f(\lambda _{\hat g_{t,s}}(p)) \text{ for all } p\in \s^n_+ . $$

Thus, if $S_{t,s_0} ^+$ intersects $\Sigma _t$ at an interior point, this contradicts the strong maximum principle (see Lemma \ref{SMP} in the Appendix). Observe that we do not really need that both hyperbolic support functions are positive. To overcame this we can either dilate at the beginning with a $t$ big enough so that $\rho _t >0$ or translate $\Sigma _t$ and $S_{t,s_0}$ at distance $\abs{s_0}$ using $T_{\abs{s_0}}$. Then, the new hyperbolic support functions are positive, they coincide at some point in the interior and differ along the boundary. All these conditions follow since $T_{\abs{s_0}}$ is an isometry. 

Therefore, it remains the case that $S_{t,s_0} ^+$ intersects $\Sigma _t$ at a boundary point. Since in this case  $s_0 =0$, the argument above shows that $\rho _t \geq t $ on $\overline{\s^n _+} $. This inequality follows since $\rho _t \geq \hat \rho _{t,s}$ on $ \overline{\s ^n _+}$ for all $s<0$, taking $s \to 0$ one can easily see that $\hat \rho _{t,s} \to \hat \rho _t := t $. 


If $\partial \Sigma _t \cap P = \emptyset$, then $S_{t,s} \cap \partial \Sigma _t = \emptyset$ for all $s \in \R$. Hence, we can translate $S_t$ up to the north pole until we find a first contact point with $\Sigma _t$, such point must be an interior point. However, as above, this contradicts the strong maximum principle. 

Therefore,  by Claim A, there exists $x \in \partial \mathbb S^n _+$ so that 
$$ \frac{\partial \hat \rho _t}{\partial \nu} (x)=0,$$hence, by the Hopf Lemma (cf. Lemma \ref{HL} in the Appendix), we obtain that $\rho _t \equiv t$ in $\overline{\s ^n _+}$. Thus, $g_t = \tilde g _t$ and hence, $g= g_0$.
\end{proof}

The same ideas work on geodesic balls in $\mathbb S ^n$ of radius $r< \pi/2$. However, in this situation we must impose an extra condition on the mean curvature along the boundary. Geometrically, in the previous result we compared $\Sigma _t$ with the semi-sphere $S_t ^+$. Now, we are going to compare with a smaller spherical cap of $S_t$ that depends on $r$. 

First, observe that the geodesic ball $D({\bf n}, r) \subset (\s ^n , g_0)$ of radius $r$ centered at the north pole satisfies that $\partial D({\bf n}, r)$ is isometric to $\s ^{n-1} (\sin (r))$ and the mean curvature of $\partial D({\bf n}, r)$ with respect to the inward orientation is $\cot (r) $. 

Second, the horospherical metric associated to the geodesic sphere $S_t \subset \mathbb H ^{n+1}$ centered at the origin (in the Poincar\'e ball Model) of radius $t$ is just the dilated metric $\tilde g _t = e^{2\tilde \rho _t} g_0= e^{2t} g_0 $ and, from the representation formula \eqref{formula}, it is parametrized by 
$$ \tilde \phi _t (x) = (\cosh (t) , \sinh (t) \, x)  \text{ for all } x \in \s ^n .$$ 

In particular,  $$H_{\tilde g _t} (x) = e^{-t} \cot (r) \textrm{ for all } x \in \partial D ({\bf n}, r).$$

Now, let $P _r$ be the totally geodesic hyperplane in $\mathbb H ^{n+1}$ whose boundary at infinity coincides with the boundary of $D({\bf n} ,r)$, that is, $\partial _\infty P_r = \partial D ({\bf n}, r)$. Set $S^+_{r,t} = \tilde \phi _t (\overline{D({\bf n},r)})$. Hence, with the conditions above (as we have already done) we can check that 
$$ \tilde \phi _t (x) \in \mathcal H (x,t) \cap P_r({\rm arc}\sinh (-e^{-t}\cot (r))), \text{ for all } x \in \partial D ({\bf n}, r) $$and $S^+_{r,t} \subset \overline{P_r({\rm arc}\sinh (-e^{-t}\cot (r))) } $   

Denoting  by $\mathcal B (x, t)$ the open horoball determined by $\mathcal H (x,t)$ we observe that 
$$
\mathcal D (a) := P_r({\rm arc}\sinh (-e^{-t}\cot (r))) \setminus  \bigcup _{x \in \partial D({\bf n} ,r)} \mathcal B (x, t)
$$ 
is a closed ball in $P_r({\rm arc}\sinh (-e^{-t}\cot (r)))$ of radius $a>0$ depending on $r$ and $t$ and centered at $q_0 = P_r({\rm arc}\sinh (-e^{-t}\cot (r))) \cap \gamma (\R) $, where $\gamma$ is the complete geodesic in $\mathbb H ^{n+1}$ joining the south and north poles. Let $\bar a >0$ the unique positive number so that 
$$ \mathcal C (\bar a)  \cap P_r({\rm arc}\sinh (-e^{-t}\cot (r))) = \partial \mathcal D(a) \subset P_r({\rm arc}\sinh (-e^{-t}\cot (r))) ,$$where  $\mathcal C (\bar a)$ is the hyperbolic cylinder in $\mathbb H ^{n+1}$ of axis $\gamma$ and radius $\bar a$, i.e., those points at distance $\bar a$ from $\gamma$.

The exact value of $\bar a$ is not important. However it can be computed explicitly. The important observation is the following. Let $P_r({\rm arc}\sinh (-e^{-t}\cot (r))) ^-$ be the halfspace determined by $P_r({\rm arc}\sinh (-e^{-t}\cot (r)))$ containing the south pole at its boundary at infinity, then 
\begin{equation}\label{OutCylinder}
 \mathcal C (\bar a) \cap P_r({\rm arc}\sinh (-e^{-t}\cot (r)))^- \cap \mathcal H (x ,t) = \emptyset \text{ for all } x \in \partial D ({\bf n} ,r) .
\end{equation}

Let $\hat g_{t}$ be the restriction of $\tilde g _{t}$ to $\overline{D({\bf n},r)}$, i.e., $\left.\tilde g _{t}\right. _{|D({\bf n},r)} = \hat g_{t}$, and $\hat \rho _{t}$ the restriction of $\tilde \rho _{t}$ to $\overline{D({\bf n},r)}$. Then, it holds
\begin{equation}\label{BConditions}
\hat \rho _{t} = t \text { and } \frac{\partial \hat \rho _{t}}{\partial \nu} = 0 \text{ on } \partial D({\bf n},r) ,
\end{equation}where $\nu$ is the inward normal along $\partial D({\bf n},r)$.

After the proof of Theorem \ref{Th:rsmall} we will explain, geometrically, the necessity on the condition for the mean curvature.

\begin{theorem}\label{Th:rsmall}
Let $(f, \Gamma)$ be an elliptic data and let $g=e^{2\rho}g_0$,  $\rho \in C^2 (\overline{\mathbb S^n_+})$, be a supersolution to $(f,\Gamma)$ in the closed  hemisphere $\overline{\mathbb S^n_+}$, i.e., 
$$f(\lambda(p)) \geq 1, \, \, \lambda _g(p)\in\Gamma \text{ for all } p\in \mathbb{S}^n_+ .$$

Assume that the boundary $\partial \s ^n_+$ with respect to $g$ is umbilic with mean curvature $H_g\geq {\rm cot}(r)$ and isometric to  $\mathbb{S}^{n-1}(\sin r)$ for some $r \in (0,\pi/2)$, here $\mathbb{S}^{n-1}(\sin r)$ denotes the standard sphere of radius $\sin r$.

Then, there exists a conformal diffeomorphism $\Phi \in {\rm Conf}(\s^n)$ so that $(\overline{\mathbb S^n_+} , \Phi^* g )$ is isometric $D ({\bf n}, r)$, where $D ({\bf n}, r)$ is the geodesic ball in $\s ^n$ with respect to the standard metric $g_0$ centered at the north pole ${\bf n}$ of radius $r$.
\end{theorem}

\begin{proof}[Proof of Theorem \ref{Th:rsmall}]
Using Obata's Theorem in this case, up to a conformal diffeomorphism, we can assume that $ g = e^{2\rho} g_0 $ is defined on $\overline{D({\bf n}, r)}$ and it is so that $\rho = 0$ on $\partial D ({\bf n} ,r)$. Moreover, the mean curvature of $\partial D ({\bf n} ,r)$ with respect to $g$ is given by
\begin{equation}\label{MB2}
{\rm cot}(r)\leq H_g:= -e^{-\rho}\frac{\partial \rho}{\partial \nu} + \cot(r) \text{ on } \partial D ({\bf n} ,r) .
\end{equation}

Now, as we have done above, we fix $t>0$ such that  the eigenvalues of the Schouten  tensor of  $g_t=e^{2(\rho +t)}g_0$ satisfy $\lambda^t_i(x)<1/2$, for all $x\in \overline{\mathbb S^n_+}$ and we denote by
 $\Sigma _t=\phi _t(\mathbb S^n_+) \subset \mathbb H^{n+1} \subset \mathbb L ^{n+2}$  the compact embedded horospherically concave hypersurface with boundary given by the representation formula (\ref{formula}) associated to $\rho _t =\rho +t$. In particular, $\rho = t$ along $\partial D ({\bf n},r)$.

As we have seen above, we have $\phi _t(x) \in \mathcal H(x,\rho_t(x))$, where $\mathcal H(x,\rho_t(x))$ is the horosphere whose point at infinity is  $x$ and distance to the origin is $t$. Moreover, the mean curvature $H_g (x)$ measures the equidistant where $\phi _t (x )$ is contained, that is 
$$ \phi _t(x) \in  \mathcal H(x,\rho_t(x)) \cap  P_r({\rm arc}\sinh (-e^{-t}H_g(x))).$$ In particular, $ \partial \Sigma _t \subset  \overline{P _r({\rm arc}\sinh (-e^{-t}\cot (r)))^- }$. Hence,

\begin{quote}
{\bf Claim:} {\it $\partial \Sigma _t$ lies outside the interior of $\mathcal C (\bar a)$. Moreover, $\phi _t (x) \in \partial \Sigma _t \cap \mathcal C (\bar a) $ for some $x \in \partial D({\bf n} ,r)$ if, and only if, $\dfrac{\partial \rho _t}{\partial \nu} =0$ at $x \in \partial D({\bf n} ,r)$.}
\end{quote}
\begin{proof}[Proof of Claim]
From \eqref{OutCylinder}, the boundary $\Sigma _t$ lies outside the interior of $\mathcal C (\bar a)$. Moreover, $\partial \Sigma _t$ touches $\mathcal C (\bar a)$ at $\phi _t (x)$ for some $x \in \partial D({\bf n} ,r)$ if, and only if, $H_g (x)  = \cot(r)$ and, from \eqref{MB2}, this is equivalent to $\frac{\partial \rho _t}{\partial \nu} =0$ at $x \in \partial D({\bf n} ,r)$. This finishes the proof of Claim.
\end{proof}

Now, consider the metric $\hat g _t = e^{2 \hat \rho _t}$ on $D({\bf n},r)$ defined above and satisfying \eqref{BConditions}. Now, we only have to compare $\rho _t$ and $\hat \rho _t$ the same way we did in Theorem \ref{Th:Hemi} and we conclude that $\rho _t \equiv \hat \rho _t$ on $D({\bf n} ,r)$. This proves the theorem.

\end{proof}

The condition on the mean curvature is fundamental to ensure that $\partial \Sigma _t$ does not touch the interior of $\mathcal C(\bar a)$. If, at some point $x\in \partial D({\bf n}, r)$, the mean curvature were smaller than $\cot (r)$, the point $\phi _t (x)$ might be in the interior of $\mathcal C (\bar a)$. Hence, when we compare $\Sigma _t$ and the spherical cap, the first contact point could be an interior point of the spherical cap and a boundary point of $\partial \Sigma _t$ and hence, we can not apply the maximum principle.

Finally, we establish our main result in this section:

\begin{theorem}\label{Th:General}
Let $p_i \in \s ^n$ and $\epsilon _ i >0 $, $i=1,\ldots , k$, be so that the closed geodesic balls $\overline{D(p_i,\epsilon _i)}\subset \s^n$ are pairwise disjoint. Set $\Omega := \s^n \setminus \bigcup _{i=1}^k D(p_i,\epsilon _i)$ and let $\Lambda \subset \Omega$ be a closed subset with empty interior.

Let $(f, \Gamma)$ be an elliptic data and let $g=e^{2\rho}g_0$,  $\rho \in C^2 (\overline{\Omega}\setminus \Lambda )$, be a supersolution to $(f,\Gamma)$ in  $\Omega \setminus \Lambda $, i.e., 
$$f(\lambda(p)) \geq 1, \, \, \lambda _g(p)\in\Gamma \text{ for all } p\in \Omega \setminus \Lambda  .$$

Assume that $g$ is complete in $\overline{\Omega}\setminus \Lambda $ and the Schouten tensor of $g$ is bounded.

Assume that each boundary component $\partial D(p_i,\epsilon _i)$ with respect to $g$ is umbilic with mean curvature $H_g\geq {\rm cot}(r)$ and isometric to  $\mathbb{S}^{n-1}(\sin (r))$ for some $r \in (0,\pi/2]$, here $\mathbb{S}^{n-1}(\sin (r))$ denotes the standard sphere of radius $\sin (r)$.

Then, there exists a conformal diffeomorphism $\Phi \in {\rm Conf}(\s^n)$ so that $(\overline{\Omega}\setminus \Lambda , \Phi^* g )$ is isometric $\overline{D ({\bf n}, r)}$, where $D ({\bf n}, r)$ is the geodesic ball in $\s ^n$ with respect to the standard metric $g_0$ centered at the north pole ${\bf n}$ of radius $r$.
\end{theorem}

The condition on $\Lambda$ having empty interior is superfluous. Under the conditions above, following ideas contained in \cite{BEQ}, one can prove that $\Lambda $ must have empty interior. 

After the proof of Theorem \ref{Th:General} we will explain the necessity of $H_g \geq 0$ when $(\partial \s ^n _+ ,g)$ is isometric to $\s ^{n-1}$ in the case of multiple boundary components, in contrast to Theorem \ref{Th:Hemi}.

\begin{proof}[Proof of Theorem \ref{Th:General}]

Since $|{\rm Sch}_g|<+\infty$ and $g$ is a complete metric, following the results in \cite{BEQ} (see also \cite{BQZ}), there exists $t>0$ such that the horospherically concave hypersurface associated 
\[
\Sigma_t = \phi_t(\Omega\setminus \Lambda) \subset \mathbb{H}^{n+1}
\]
is properly embedded with boundary and $\partial _{\infty}\Sigma_t = \Lambda$. Without loss of generality we can assume that $\Sigma _t$ is locally convex with respect to the canonical orientation $\eta _t$ by taking $t$ big enough.

Observe that, up to a conformal diffeomorphism $\Phi \in {\rm Conf} (\mathbb S^n)$, we can assume that one connected component of $\partial \Omega$ is $\partial \mathbb{S}^{n}_+$. Consider the case where $(\partial \Omega, g)$ is isometric to $(\mathbb{S}^{n-1},g_0)$. The case where  $(\partial \Omega, g)$ is isometric to $(\partial \mathbb{D}(r),g_0)$ is analogous. Observe that at the beginning of Theorem \ref{Th:Hemi} we did a conformal transformation to ensure that $\rho = 0$ along $\partial \mathbb{S}^{n}_+$. We can do this to ensure $\rho = 0$ along one connected component of the boundary (of course, not all of them). We assume $\Gamma_1=\partial \mathbb{S}^{n}_+$ has this property. Observe that after applying this conformal diffeomorphism we can assume $\Phi (\Omega) \subset \mathbb{S}^{n}_+$. Now consider the half-sphere $S^+_t\subset P^+$ as in the Theorem \ref{Th:Hemi}. We only need to prove that $S^+_t$ does not touch any other boundary component. 

As we did in Theorem \ref{Th:Hemi}, consider the hyperbolic translation $T_s : \mathbb H ^{n+1} \to \mathbb H ^{n+1}$ and set $S^+_{t,s} = T_s (S^+_t)$. Then, there exists $s_0 <0$ so that $\Sigma _t \cap S^+_{t,s} = \emptyset$ for all $s \geq s_0$. Then, we increase $s$ to $0$ up to the first contact point with $\Sigma _t$. If this first contact point happens either at interior points or at boundary points for $s=0$, then $\Sigma _t$ equals $S^+_t$ by the maximum principle as we did in Theorem \ref{Th:Hemi}.

Therefore, we only must show that the first contact point does not occur at an interior point of $S^+ _{t,s}$, for some $s\in [s_0 ,0]$, and a boundary point of $\Sigma _t$. Assume this happens, and let $\Gamma_2$ be the other boundary component of $\partial\Omega\setminus \Gamma_1$ that $S^+_{t,s}$ touch.

Let $p \in \s ^n _+$ and $\epsilon >0$, $\overline{D(p,\epsilon)}\subset \s ^n _+$, so that $\Gamma _2 = \phi _t (\partial D(p,\epsilon))$. Let $P$ and $Q$ be the totally geodesic hyperplanes in $\mathbb H ^{n+1}$ whose boundaries at infinity are 
$$ \partial _\infty P = \partial \s ^n_+ \text{ and } \partial _\infty Q = \partial D(p, \epsilon) . $$

Let $Q^+ $ be the halfspace determined by $Q$ whose boundary at infinity contains $p \in \partial _\infty Q^+$. Let $\phi _t (x) =q \in \Gamma _2 \cap S^+_{t,s}$ be a first contact point. Let $\eta _t$ and $\tilde \eta _{t,s}$ be the canonical orientation of $\Sigma _t$ and $S^+_{t,s}$ respectively. Then, $\phi _t (x) \in Q({\rm arc}\sinh (e^{-t}H_g(x)))$, where $Q({\rm arc}\sinh (e^{-t}H_g(x)))$ is the equidistant to $Q$ at distance ${\rm arc}\sinh (e^{-t}H_g(x))$ contained in $Q^+$.

Since we are assuming that the mean curvature $H_{g}$ is non-negative along $\partial D(p,\epsilon)$ we have that $\eta _t (q)$ points towards $\overline{Q^+}$, it could belong to the tangent bundle of $Q$ if $H_g (x)=0$, $q = \phi _t (x)$. Now, since $S^+_{t,s}$ is convex with respect to the canonical orientation $\tilde \eta _{t,s}$, then $\tilde \eta _{t,s} (q)$ points towards $Q^-$. Since we are assuming that $q$ is the first contact point, the only possibility is that $\eta _t (q) = - \tilde \eta _{t,s} (q)$. However, if this were the case, since $\Sigma _t$ and $S^+ _{t,s}$ are locally convex and their tangent hyperplanes coincide, they must be (locally) in opposite sides of the tangent hyperplane, in other words, $S^+_{t,s}$ is approaching by the concave side of $\Sigma _t$, which is a contradiction. Hence, in any case, the first contact point does not occur at an interior point of $S^+ _{t,s}$, for some $s\in [s_0 ,0]$, and a boundary point of $\Sigma _t$.

Thus, this finishes the proof of Theorem \ref{Th:General}


\end{proof}

Observe that the condition $H_g \geq 0$ is essential in Theorem \ref{Th:General}, in contrast to Theorem \ref{Th:Hemi}. The reason is that this condition gives us a direction of the canonical orientation $\eta _t$ at the contact point. If the mean curvature at some point were negative, both $\eta _t$ and $\tilde \eta _{t,s}$ point toward the same halfspace $Q^-$ at the contact point $q$, and we can not achieve a contradiction.

\section{Proof of Theorem A}

Now, we are ready to prove our main result. For simplicity, we divide the proof into two cases.

\subsection{$M$ is simply-connected}
\begin{proof}
First, we prove our Theorem A under the condition that $M$ is simply-connected. In this case, there exists a developing map $\Psi : M \rightarrow \s ^n$. Since $\partial M$ is umbilic, and to be umbilic is a conformal invariant, the image of $\partial M$ must be umbilic in $\s ^n$. Hence, $\partial M$ is contained in a hypersphere $\mathcal S \subseteq \s ^n$. Note that, in fact, $\Psi _{|_{\partial M}}: \partial M \rightarrow \mathcal S$ is a diffeomorphism. Composing it with a conformal diffeomorphism of $\s ^n$, if necessary, we can assume that $\mathcal S$ is the equator $\partial \s _+^n = \{x \in \s ^n\,; \,\,x_{n+1} = 0\}$. Now, consider the double manifold $\hat{M}=M\bigcup\limits_{\partial M}(-M).$ We are writing $-M$ for the second copy of $M$ in $\hat{M}$ in order to distinguish it from $M$ itself. We extend $\Psi$ to a map $\hat{\Psi}: M\rightarrow \s ^n$ in a natural way: we write $\Psi = (\Psi_1 , \ldots , \Psi_{n+1} )$ and set
$$\begin{displaystyle}
  \hat{\Psi}(x) := 
  \begin{cases}
    \Psi(x) & \text{if } x\in M \\
    (\Psi_1(x), \ldots , \Psi_{n+1}(x),-\Psi_{n+1}(x) ), &  \text{if } x\in -M
  \end{cases}
\end{displaystyle}$$
Then $\hat{\Psi}$ is well-defined and continuous because $\Psi_{n+1}(x) = 0$ for $x \in \partial M$. Moreover, it is a local homeomorphism. It follows that $\hat{\Psi}$ is a homeomorphism and hence $\Psi$ is injective. Furthermore, the image is either $\s ^n_+$ or $\s ^n _-$. Let $\{{\bf s}, {\bf n} =-{\bf s}\}$ be a pair of antipodal points. By composing $\Psi$ with a conformal diffeomorphism of $\s^n$, we may assume that the image of $\Psi$ is $\overline{\s ^n _+}$. 

Now, we can pushforward the metric $g$ on $M$ to $\overline{\s^n _+}$ via $\Psi$, $\tilde g = (\Psi ^{-1})^*g$, and we obtain a conformal metric to standard metric on the sphere satisfying that the boundary $\partial \s ^n_+$ with respect to $\tilde g$ is umbilic with mean curvature $H_{\tilde g}\geq {\rm cot}(r)$ and isometric to  $\mathbb{S}^{n-1}(\sin r)$ for some $r \in (0,\pi/2]$, here $\mathbb{S}^{n-1}(\sin r)$ denotes the standard sphere of radius $\sin r$. Therefore, either Theorem \ref{Th:Hemi} if $r =\pi/2$ (in this case we do not need to assume $H_g \geq 0$) or Theorem \ref{Th:rsmall} if $r\in (0, \pi/2)$ imply that $\tilde g$ is isometric (up to a conformal diffeomorphism) to $\overline{D({\bf n},r)}$. This concludes the proof of Theorem A in the simply-connected case.
\end{proof}

\subsection{$M$ is not simply-connected}

In this case, we will use Theorem \ref{SP}. Then, there exists an injective conformal diffeomorphism $\Psi : M \to \Omega \setminus \Lambda $ where $\Omega = \Omega (\epsilon_i , p_i ) := \s ^n\backslash \left( \bigcup\limits_{i} D(p_i ,\epsilon_i)\right)$, $D(p_i ,\epsilon_i)$ are geodesic balls in $\s^n$ centered at $p_i$ of radius $\epsilon _i$ with disjoint closures and $\Lambda$ is a closed subset of Hausdorff dimension at most $\frac{n-2}{2}$.

Hence, as we did above, we can push forward the metric on $M$ to $\Omega \setminus \Lambda $ as $\tilde g = (\Psi ^{-1}) ^*g$, $\tilde g$ is conformal to the standard metric on the sphere. This metric is complete (cf. \cite[Section 2]{LiNg}) and its Schouten tensor is bounded, since the Schouten tensor of $g$ is bounded in $M$. Moreover, the boundary conditions on $g$ imply that each boundary component $\partial D(p_i,\epsilon _i)$ with respect to $\tilde g$ is umbilic with mean curvature $H_{\tilde g}\geq {\rm cot}(r)$ and isometric to  $\mathbb{S}^{n-1}(\sin (r))$ for some $r \in (0,\pi/2]$, here $\mathbb{S}^{n-1}(\sin (r))$ denotes the standard sphere of radius $\sin (r)$.

Therefore, Theorem \ref{Th:General} implies that there exists a conformal diffeomorphism $\Phi \in {\rm Conf}(\s^n)$ so that $(\overline{\Omega}\setminus \Lambda , \Phi^* \tilde g )$ is isometric $D ({\bf n}, r)$, where $D ({\bf n}, r)$ is the geodesic ball in $\s ^n$ with respect to the standard metric $g_0$ centered at the north pole ${\bf n}$ of radius $r$. In particular, $\Lambda = \emptyset$ and the number of connected components at the boundary is one. This implies that $M$ is simply connected via $\Psi$. This concludes the proof of Theorem A.

\section{Rigidity for hypersurfaces in $\mathbb H ^{n+1}$}

Now, we will see how our results on Section \ref{hemisphere} apply to  hypersurfaces $\Sigma$ in $\Hy ^{n+1}$. We are going to establish here a simplified version of that we could, but which is geometrically more appealing. 

First, we define the geometric setting. Let $P_i \subset \Hy ^{n+1}$, $i=1,\ldots, m$, be pairwise disjoint totally geodesic hyperplanes and let $\mathcal O (m)$ be the connected component of $\Hy^{n+1}\setminus \bigcup _{i=1}^m P_i$ whose boundary is $\partial \mathcal O (m) = \bigcup _{i=1}^m P_i$. Fix $r\geq 0$ and denote by $P_i(r)$ the equidistant hypersurface $P_i$ at distance $r$ so that $P_i (r) \subset \Hy^{n+1}\setminus \mathcal O$. Assume that $P_i (r)$, $i=1, \ldots , m$, are pairwise disjoint and denote  by $\mathcal O (m, r)$ the connected component of $\Hy^{n+1}\setminus \bigcup _{i=1}^m P_i (r)$ whose boundary is $\partial \mathcal O (m, r) = \bigcup _{i=1}^m P_i (r)$. Observe that the boundary at infinity  $\overline{\Omega(m)} := \partial _\infty \mathcal O (m,r) \subset \s ^n $ satisfies that  $\partial \Omega (m)= \bigcup _{i=1}^m \partial D(p_i ,\epsilon _i)$, for certain $p_i \in \s^n$ and $\epsilon _i >0$. Moreover, we orient each $P_i$ so that the normal $N_i$ along $P_i$ points into $\mathcal O (m,r)$. A domain $\mathcal O (m,r)$ in the above conditions is called a {\it $(m,r)-$domain}.

Second, we define how the hypersurface $\Sigma$ sits into a $(m,r)-$domain. Let $\Sigma \subset \Hy ^{n+1}$ be a properly embedded hypersurface with boundary. We say that $\Sigma $ {\it sits into a $(m,r)-$domain}, denoted by $\Sigma \subset \mathcal O (m,r)$, if 
\begin{itemize}
\item $\Sigma \setminus \partial \Sigma \subset \mathcal O (m,r)$, 
\item $\partial \Sigma = \bigcup _{i=1}^m \mathcal S _i$, where each $\mathcal S _i $ is homeomorphic to $\s ^{n-1}$ and $\mathcal S _i \subset P _i (r)$, 
\item let $\mathcal D _i \subset P_i$ the domain bounded by $\mathcal S _i$ in $P_i (r)$, the orientation $\eta$ of $\Sigma$ is the one pointing into the domain $W \subset \Hy ^{n+1}$ bounded by $\Sigma \cup \left( \bigcup_{i=1}^m \mathcal D _i \right)$, and 
\item $\partial _\infty \Sigma \subset \Omega (m)$.
\end{itemize}

Third, we set the type of elliptic inequality the hypersurface will satisfy. We recall  the definition of elliptic data for a hypersurface in $\mathbb H ^{n+1}$ (cf. \cite[Section 4]{BEQ} and references therein). Let 
$$\Gamma ^* _n =\{ (x_1, \ldots , x_n) \in \real ^n \, : \,\, x_i >1 \}$$and
$$\Gamma ^* _1 =\{ (x_1, \ldots , x_n) \in \real ^n \, : \,\, \sum _{i=1}^nx_i >n  \}.$$

Consider a symmetric function $\mathcal W (x_1 , \ldots , x_n)$ with $\mathcal W (1, \ldots , 1) = 0$ and $\Gamma ^*$ an open connected component of 
$$ \{ (x_1, \ldots , x_n) \in \real ^n \, : \,\, \mathcal W (x_1, \ldots ,x_n ) >0  \}. $$

We say that $(\mathcal W , \Gamma ^* , \kappa _0) $, $\kappa_0 >0$, is an elliptic data if they satisfy

\begin{enumerate}
  \item $\Gamma ^*_n \subset \Gamma ^* \subset \Gamma ^* _1$,
  \item $\mathcal W$ is symmetric,
  \item $\mathcal W>0$ in $\Gamma^*$,
  \item $\mathcal W |_{\partial\Gamma^*}=0$,
  \item $\dfrac{\partial \mathcal W}{\partial x_i}>0$ for all $ i=1\ldots, n.$,
  \item $\mathcal W (\kappa_0, \ldots , \kappa_0) =1$.
\end{enumerate}

Then, given an elliptic data $(\mathcal W, \Gamma ^* ,\kappa _0)$ we say that an oriented hypersurface $\Sigma \subset \Hy^{n+1}$ is a {\it supersolution} to $(\mathcal W, \Gamma ^* ,\kappa _0)$ if 
$$ \mathcal W (k (p)) \geq 1, \, \, k (p) \in \Gamma ^* \text{ for all } p\in \Sigma ,$$where $k(p):= (k_1 (p), \ldots, k_n (p))$ is composed by the principal eigenvalues of $\Sigma$ at $p\in \Sigma$ with respect to the chosen orientation. 

We have already established the geometric configuration. In order to state appropriately our main result, we need to introduce some notation. 

Fix $r \geq 0$ and $\kappa _0 >1$. Let $S_p(\kappa _0)$ be the totally umbilic geodesic sphere centered at  $p\in \Hy^{n+1}$ whose principal curvatures (with respect to the inward orientation) are equal to $\kappa _0$. Let $P(r)$ be a equidistant hypersurface to a totally geodesic hyperplane $P$. Denote by $P(r)^+$ the convex component of $\Hy^{n+1}\setminus P$. Let $p_r \in \Hy ^{n+1}$ be a point so that $S(\kappa_0,r)^+ := S_{p_r} (\kappa_0) \cap P(r)^+$ makes a constant angle $\alpha (r)={\rm arc}\cos\left( -\frac{r}{\sqrt{1+r^2}}\right)$, the angle here is measure between the inward normal along the geodesic sphere and the normal along $P(r)$ pointing into the convex side. 

\begin{definition}
We say that $\Sigma$ is a {\it $(\kappa_0,r)-$spherical cap} if $\Sigma := S(\kappa_0,r)^+$, up to an isometry of $\Hy^{n+1}$.
\end{definition}

Recall that the inradius of a closed embedded hypersurface $\mathcal S$ in $P(r)$, denoted by ${\rm InRad}(\mathcal S, P(r))$, is the radius of the biggest geodesic ball in $P(r)$ contained in the domain bounded by $\mathcal S$ in $P(r)$. Then, we set $\imath (\kappa _0 ,r) := {\rm InRad}(\partial S(\kappa _0 ,r)^+ ,P(r)) >0$.

It is clear that $(\kappa_0,r)-$spherical caps will be the model hypersurfaces to compare with in the next result.

\begin{theorem}\label{capillar}
Fix $m \in \mathbb N \cup \set{0} $ and $r \geq 0$. Consider a $(m,r)-$domain $\mathcal O (m,r)$ and let $\Sigma \subset \mathcal O (m,r)$ be a properly embedded hypersurface sitting on it. 

Let $(\mathcal W, \Gamma ^* ,\kappa _0)$ be an elliptic data and assume that $\Sigma$ is a supersolution to $(\mathcal W, \Gamma ^* ,\kappa _0)$. Assume that along the boundary $\Sigma$ satisfies:
\begin{itemize}
\item $ \meta{\eta  (x)}{N_i(x)} \leq  -\frac{r}{\sqrt{1+r ^2}}$ for each $x\in \mathcal S_i $.
\item ${\rm InRad}(S_{i_0} ,P_{i_0}(r)) \geq \imath (\kappa _0 ,r) $ for some $i_0 \in \set{1,\ldots ,m}$.
\end{itemize}

Then, $\Sigma $ is, up to an isometry of $\Hy^{n+1}$, a $(\kappa_0,r)-$spherical cap.
\end{theorem}

\begin{proof}[Proof of Theorem \ref{capillar}]
The proof follows from the arguments given in Theorem \ref{Th:General}. In this case, we only need to compare with the $(r,\kappa_0)-$spherical cap.
\end{proof}

\begin{remark}
We can drop the embeddedness hypothesis on Theorem \ref{capillar}, as far as $\Sigma \cup \left( \bigcup_{i=1}^m \mathcal D _i \right)$ is Alexandrov embedded.  
\end{remark}

\section{Toponogov type theorem}

In this section, we proceed as Espinar-G\'alvez-Mira \cite{EGM} in order to define the Schouten tensor for a two-dimensional domain endowed with a metric $g$ conformal to the standard metric $g_0$ on $\mathbb{S}^2$. Consider $g=e^{2\rho}g_0$, where $\rho\in C^{2}(\Omega)$, defined on a domain $\Omega \subset \mathbb{S}^2$. In this case, we define the Schouten tensor ${\rm Sch}_g$ of $g$ from the following relation:
$$
{\rm Sch_g} +\nabla ^2\rho + \frac{1}{2}\|\nabla \rho \|^2g_0={\rm Sch}_{g_0} + \nabla\rho \otimes \nabla \rho
$$
where $\nabla$ and $\nabla ^2$ are the gradient and the hessian with respect to the metric $g_0$, respectively, and $\| \cdot \|$ denote the norm with respect of $g_0$. Consider then $\lambda_g=(\lambda_1,\lambda_2)$, where $\lambda_i$, $i=1,2$, are the eigenvalues of the Schouten tensor given by the expression above. Note that if $f(x,y)=x+y$ then
\[
f(\lambda_1,\lambda_2)=\frac{R_g}{2(n-1)}=K\,,
\]
since $n=2$, where $K$ is the Gaussian curvature of $g=e^{2\rho}g_0$. Then the Liouville problem (i.e.  the Yamabe problem in dimension $n=2$) is a particular problem of more general elliptic problems for conformal metrics in $\mathbb{S}^2$. Moreover, we can consider the Min-Oo conjecture for more general elliptic problems and see Toponogov's Theorem as a particular case of it. 

Of particular interest is when we consider the product of the eigenvalues, i.e., $f(x,y) = \sqrt{xy}$. It is clear that $(f,\Gamma _2)$ is an elliptic data and, if we consider $g=e^{2\rho} g_0$, $\rho \in C^2 ( \Omega )$, that satisfies $f(\lambda _g) \geq 1$, then $\rho$ is a super-solution to the Monge-Amp\`ere type equation 
$$ e^{-4\rho}{\rm det}_{g_0} \left(\nabla ^2\rho - \nabla\rho \otimes \nabla \rho - \frac{1}{2}(1- \|\nabla \rho \|^2)g_0 \right) \geq 1 .$$

That is the subject of our next result.

\begin{theorem}
Let $(f,\Gamma) $ be an elliptic data. Let $(M^2,g)$ be a compact surface with smooth boundary such that $f(\lambda_g)\geq1$. Suppose the geodesic curvature $k$ and the length $L$ of the boundary $\partial M$ (w.r.t. $g$) satisfy $k\geq c\geq0$ and $L=\frac{2\pi}{\sqrt{1+c^2}}$ respectively. Then $(M^2,g)$ is isometric to a disc of radius $r=cot^{-1}(c)$ in $\mathbb{S}^2$.
\end{theorem}
\begin{proof}
Since $(f,\Gamma) $ is elliptic we have that $K>0$, where $K$ is the Gaussian curvature of $(M,g)$. Hence, since the geodesic curvature $k$ of the boundary satisfies $k\geq c\geq0$, it follows from the Gauss-Bonnet formula that 
$$
2\pi\chi(M)=\int_{M}Kdv_{M}+\int_{\partial\Sigma}kds>0\,,
$$
where $\chi(M)$ is the Euler number of $M$. Therefore $M$ is a disc. By the Riemann mapping theorem, $(M^2,g)$ is conformally equivalent to the unit disc $\mathbb{D}=\{(x,y)\in \mathbb{R}^2\,:\,\,x^2+y^2 \leq 1\}$ with the flat metric $ds_0^2$. Without loss of generality, we can write $g=e^{2\rho}g_0$, with $\rho\in C^{2}(M)$, and $M=\overline{\mathbb{S}^2_+}$, where $g_0$ denote the standard metric on $\mathbb{S}^2_+$, since $(\mathbb{D}, ds_0^2)$ is conformally equivalent to ($\mathbb{S}^2_+,g_0)$. Moreover, $\rho$ satisfies 
\[
\frac{\partial \rho}{\partial \nu}=-ke^{\rho}\leq -ce^{\rho}\,.
\]

Moreover, since $L = \frac{2\pi}{\sqrt{1+c^2}}$, we can reparametrize $\overline{\mathbb{S}^2_+}$ so that $\rho = - \ln \sqrt{1+c^2}$. Now, arguing as in the proof of Theorems \ref{Th:Hemi} and \ref{Th:rsmall}, we obtain that $(M^2,g)$ is isometric to a disc of radius $r={\rm arc}\cot(c)$ in $\mathbb{S}^2$.
\end{proof}

As a direct consequence of the result above, we obtain the following version of the Toponogov Theorem.

\begin{theorem}\label{Th:Toponogov}
Let $(f,\Gamma)$ be an elliptic data. Let $(M^2,g)$ be a closed surface such that $f(\lambda_g)\geq1$. Assume that there exists a simple closed geodesic in $M$ with length $2\pi$. Then $(M^2,g)$ is isometric to  the standard sphere $\mathbb{S}^2$.
\end{theorem}
\begin{proof}
Suppose that $\gamma$ is a simple closed geodesic in $M$ with length $2\pi$. We cut $M$ along $\gamma$ to obtain two compact surfaces with the geodesic  $\gamma$ as their common boundary. The result follows from applying the previous theorem to either of these two compact surfaces with boundary.
\end{proof}

\section{Appendix A: comparison principle}

In this appendix we recover some results contained in \cite{JinLiLi,LiLi1,LiLi2} to make this paper as self-contained as possible. Specifically, we will use \cite[Lemma 6.1]{JinLiLi} and its proof, that relies in the strong maximum principle and Hopf Lemma developed in \cite{LiLi1,LiLi2}. We can summarize these results as follows:

\begin{lemma}[Strong Maximum Principle]\label{SMP}
Let $(f,\Gamma) $ be an elliptic data. Let $g_i = e^{2 \rho _i} g_0$, $\rho _i \in C^2 (\Omega) \cap C^{1}(\overline \Omega)$ for $\Omega \subset \s ^n$, be two conformal metrics so that 
\begin{itemize}
\item $f(\lambda _{g_1}(p)) \geq f(\lambda _{g_2}(p)) , \,\, \lambda _{g_i} (p)\in \Gamma, \, i=1,2, \, \text{ for all } p \in \Omega$,
\item $\rho _1 , \rho _2 >0$.
\end{itemize}

If $\rho _1 - \rho _2 >0 $ on $\partial \Omega$ then $\rho _1 - \rho _2 >0$ on $\Omega$.
\end{lemma}

And 

\begin{lemma}[Hopf Lemma]\label{HL}
Let $(f,\Gamma) $ be an elliptic data. Let $g_i = e^{2 \rho _i} g_0$, $\rho _i \in C^2 (\Omega)\cap C^{1}(\overline \Omega)$ for $\Omega \subset \s ^n$, be two conformal metrics so that 
\begin{itemize}
\item $f(\lambda _{g_1}(p)) \geq f(\lambda _{g_2}(p)) , \,\, \lambda _{g_i} (p) \in \Gamma, \, i=1,2, \, \text{ for all } p \in \Omega$,
\item $\rho _1 \geq \rho _2 > 0$.
\end{itemize}

If $\frac{\partial}{\partial \eta}(\rho _1 - \rho _2 ) \leq 0 $ at $p \in \partial \Omega$ then $\rho _1 = \rho _2 $ on $\Omega$.  
\end{lemma}

We should say that the results in \cite{JinLiLi} do not need that $f$ is homogeneous of degree one. Also, in \cite{JinLiLi}, the authors assumed $f \in C^\infty (\Gamma ) \cap C^0 (\bar \Gamma)$, but it suffices $f \in C^1 (\Gamma ) \cap C^0 (\bar \Gamma)$.

\bibliographystyle{amsplain}
\bibliography{mybibfile}
\end{document}